\documentclass[a4paper,12pt]{article}
\usepackage[top=2.5cm,bottom=2.5cm,left=2.5cm,right=2.5cm]{geometry}
\usepackage{cite,amsmath,amssymb}
    \usepackage[margin=1cm,%
                font=small,%
                format=hang,%
                labelsep=period,%
                labelfont=bf]{caption}
\usepackage[algoruled,linesnumbered]{algorithm2e}
\usepackage{algorithmic}

\usepackage{amsfonts,epsf,here,graphicx}
\usepackage{latexsym,multirow,rotating}
\usepackage{pgf,tikz,subfigure}
\usepackage{epstopdf}

\newtheorem{theorem}{\bf Theorem}[section]

\newtheorem{lemma}[theorem]{\bf Lemma}
\newtheorem{proposition}[theorem]{\bf Proposition}
\newtheorem{conjecture}[theorem]{\bf Conjecture}

\newcommand{\qed}{\hfill $\square$ \bigskip}

\newcommand{\aut}{\rm Aut}

\begin{document}

\baselineskip=0.30in
\vspace*{40mm}

\begin{center}
{\LARGE \bf The Graovac-Pisanski index of armchair nanotubes}
\bigskip \bigskip

{\large \bf Niko Tratnik$^{a}$, Petra \v Zigert Pleter\v sek$^{a,b}$}

\smallskip
{\em  $^a$Faculty of Natural Sciences and Mathematics, University of Maribor,  Slovenia} \\
{\em  $^b$Faculty of Chemistry and Chemical Engineering, University of Maribor, Slovenia} \\

e-mail: {\tt niko.tratnik@gmail.com, \tt petra.zigert@um.si}

\bigskip\medskip

(\today)

\end{center}

\noindent
\begin{center} {\bf Abstract} \end{center}

\vspace{3mm}\noindent
The Graovac-Pisanski index, which is also called the modified Wiener index, considers the symmetries and the distances in molecular graphs. Carbon nanotubes are molecules made of carbon with a cylindrical structure possessing unusual valuable properties. In a mathematical model we can consider them as a subgraph of a hexagonal lattice embedded on a cylinder with some vertices being identified. In the present paper, we investigate the automorphisms and the orbits of armchair nanotubes and derive the closed formulas for their Graovac-Pisanski index. 

\baselineskip=0.30in



\section{Introduction}

Theoretical molecular descriptors are graph invariants that play an important role in chemistry, pharmaceutical sciences, etc. The most famous molecular descriptor is the Wiener index introduced in 1947 \cite{Wiener}. 

The Graovac-Pisanski index is a molecular descriptor that considers symmetries and distances in a graph. It measures how far the vertices of a graph are moved on the average by its automorphisms. The Graovac-Pisanski index was introduced by Graovac and Pisanski in 1991 \cite{graovac} under the name modified Wiener index. However, the name modified Wiener index was later used for different variations of the Wiener index \cite{gu-vu,li-li,ni-tr}. Therefore, we use the name Graovac-Pisanski index as suggested by Ghorbani and Klav\v zar in \cite{ghorbani}.

Carbon nanotubes are carbon compounds with a cylindrical structure, first observed in 1991  \cite{ii}. The extremely large ratio of length to diameter causes unusual properties of these molecules, which are valuable for nanotechnology, electronics, optics and other fields of materials science and technology. Carbon nanotubes can be open-ended or closed-ended. Open-ended single-walled carbon  nanotubes are also called tubulenes.

It was shown in \cite{ashrafi_koo_diu1} that the quotient of the Wiener index and the Graovac-Pisanski index is strongly correlated with the topological efficiency for some nanostructures. The topological efficiency was introduced in \cite{cataldo,ori} as a tool for the classification of the stability of molecules.

For recent studies on the Graovac-Pisanski index of some molecular graphs and nanostructures see also \cite{ashrafi_diu,ashrafi_koo_diu,ashrafi_sha,
koo_ashrafi3,koo_ashrafi,koo_ashrafi2,sha_ashrafi}.  Moreover, the Graovac-Pisanski index of zig-zag nanotubes was computed in \cite{tratnik}. We use similar ideas to compute this index for armchair nanotubes, but in some places our computation is more difficult and requires some additional insights.

In the present paper we first describe the automorphisms of armchair nanotubes and compute the orbits under the natural action of the automorphism group on the set of vertices of a graph. In the second part, the Graovac-Pisanski index for these nanotubes is computed. For this purpose, different cases according to the number of layers and the width of a nanotube are considered. Final results are then gathered in Table \ref{tabela5}.

\section{Preliminaries}

Unless stated otherwise, the graphs considered in this paper are finite and connected. The {\em distance} $d_G(x,y)$ between vertices $x$ and $y$ of a graph $G$ is the length of a shortest path between vertices $x$ and $y$ in $G$. We also write $d(x,y)$ for $d_G(x,y)$. Furthermore, if $S \subseteq V(G)$ and $x \in V(G)$, we define $d(x,S) = \sum_{y \in S}d(x,y)$.
\bigskip

\noindent
The {\em Wiener index} of a graph $G$ is defined as $\displaystyle{W(G) = \frac{1}{2} \sum_{u \in V(G)} \sum_{v \in V(G)} d_G(u,v)}$. Moreover, if $S \subseteq V(G)$, then $W(S) = \frac{1}{2} \sum_{u \in S} \sum_{v \in S} d_G(u,v)$.
\bigskip

\noindent  
An \textit{isomorphism of graphs} $G$ and $H$ with $|E(G)|=|E(H)|$ is a bijection $f$ between the vertex sets of $G$ and $H$, $f: V(G)\to V(H)$,
such that for any two vertices $u$ and $v$ of $G$ it holds that if $u$ and $v$ are adjacent in $G$ then $f(u)$ and $f(v)$ are adjacent in $H$. When $G$ and $H$ are the same graph, the function $f$ is called an \textit{automorphism} of $G$. The composition of two automorphisms is another automorphism, and the set of automorphisms of a given graph $G$, under the composition operation, forms a group ${\aut}(G)$, which is called the \textit{automorphism group} of the graph $G$.
\bigskip

\noindent 
\textit{The Graovac-Pisanski index} of a graph $G$, $\widehat{W}(G)$, is defined as
$$\widehat{W}(G) = \frac{|V(G)|}{2 |{\aut}(G)|} \sum_{u \in V(G)} \sum_{\alpha \in {\aut}(G)} d_G(u, \alpha(u)).$$

\noindent Next, we repeat some important concepts from group theory. If $G$ is a group and $X$ is a set, then a \textit{group action} $\phi$ of $G$ on $X$ is a function $\phi :G \times X \to X$
that satisfies the following: $\phi(e,x) = x$ for any $x \in X$ (here, $e$ is the neutral element of $G$) and $\phi(gh,x)=\phi(g,\phi(h,x))$ for all $g,h \in G$ and $x \in X$. The \textit{orbit} of an element $x$ in $X$ is the set of elements in $X$ to which $x$ can be moved by the elements of $G$, i.e. the set $\lbrace \phi(g,x) \, | \, g \in G \rbrace$. If $G$ is a graph and ${\aut}(G)$ the automorphism group, then $\phi: {\aut}(G) \times V(G) \to V(G)$, defined by $\phi(\alpha,u) = \alpha(u)$ for any $\alpha \in {\aut}(G)$, $u \in V(G)$, is called the \textit{natural action} of the group ${\aut}(G)$ on $V(G)$.

\noindent
It was shown in \cite{graovac} that if $V_1, \ldots, V_t$ are the orbits under the natural action of the group ${\aut}(G)$ on $V(G)$, then
\begin{equation}
\label{formula}
\widehat{W}(G) = |V(G)| \sum_{i=1}^t \frac{1}{|V_i|}W(V_i).
\end{equation}

\noindent
We also introduce $W'(G) = \sum_{i=1}^t W(V_i)$, which is the sum of the Wiener indices of orbits of $G$.
\bigskip

\noindent 
The \textit{dihedral group} $D_n$ is the group of symmetries of a regular polygon with $n$ sides. Therefore, the group $D_n$ has $2n$ elements. The \textit{cyclic group} $\mathbb{Z}_n$ is a group that is generated by a single element of order $n$. Given groups $G$ and $H$, the \textit{direct product} $G \times H$ is defined as follows. The underlying set is the Cartesian product $G \times H$ and the binary operation on $G \times H$ is defined component-wise: $(g_1,h_1)(g_2,h_2)=(g_1g_2,h_1h_2)$, $(g_1,h_1),(g_2,h_2) \in G \times H$.
\bigskip

\noindent
If $G$ and $H$ are groups, then a \textit{group isomorphism} is a bijective function $f: G \rightarrow H$ such that for all $u$ and $v$ in $G$ it holds $f(uv)=f(u)f(v)$.
\bigskip

Finally, we will formally define open-ended carbon nanotubes, also called tubulenes (see \cite{sa}). Choose any lattice point in the hexagonal lattice as the origin $O$. Let $\overrightarrow{a_1}$ and $\overrightarrow{a_2}$ be the two basic lattice vectors.
 Choose a vector $ \overrightarrow{OA} =n\overrightarrow{a_1}+m \overrightarrow{a_2}$
such that $n$ and $m$ are two integers and $|n|+|m|>1$, $nm\neq -1$. Draw two straight lines $L_1$ and $L_2$ passing through
$O$ and $A$ perpendicular to $O A$, respectively. By rolling up the hexagonal strip between $L_1$ and $L_2$ and gluing $L_1$ and $L_2$ such
that $A$ and $O$ superimpose, we can obtain a hexagonal tessellation $\mathcal{HT}$ of the cylinder. $L_1$ and $L_2$ indicate the direction of
the axis of the cylinder. Using the terminology of graph theory, a {\em tubulene} $T$ is defined to be the finite graph induced by all
the hexagons of $\mathcal{HT}$ that lie between $c_1$ and $c_2$, where $c_1$ and $c_2$ are two vertex-disjoint cycles of $\mathcal{HT}$ encircling the axis of
the cylinder.  The vector $\overrightarrow{OA}$ is called the {\em chiral vector} of $T$ and  the cycles $c_1$ and $c_2$ are the two open-ends of $T$. 

\noindent
For any  tubulene $T$, if its chiral vector is $ n \overrightarrow{a_1} + m \overrightarrow{a_2}$, $T$ will be called an $(n,m)$-type tubulene, see Figure \ref{nano}. If $T$ is a $(n,m)$-type tubulene where $n=m$, we call it an \textit{armchair tubulene}.

\begin{figure}[!h]
	\centering
		\includegraphics[scale=0.7, trim=0cm 0cm 1cm 0cm]{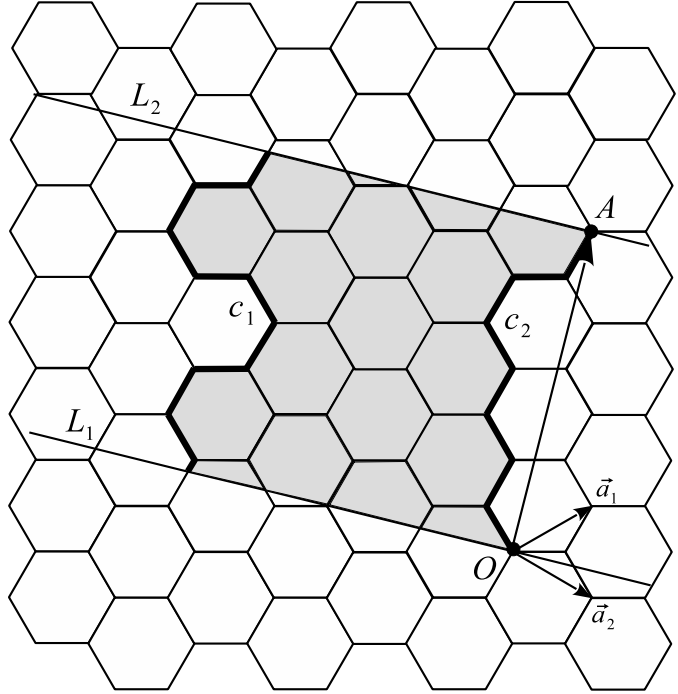}
\caption{Illustration of a $(4,-3)$-type tubulene.}
	\label{nano}
\end{figure}

\section{Armchair tubulenes and their automorphisms}
\label{sec:armchair}

 Let $T$ be an armchair tubulene such that $c_1$ and $c_2$ are the shortest possible cycles encircling the axis of the cylinder and such that there is the same number of hexagons in every column of hexagons (see Figure \ref{armchair}). If $T$ has $n$ vertical layers of hexagons, each containing exactly $p$ hexagons, then we denote it by $AT(n,p)$. Obviously, $n$ must be an even number. Note that $AT(n,p)$ is a $(\frac{n}{2},\frac{n}{2})$-type tubulene. We always assume that $n \geq 2$ and $p \geq 1$. Moreover, let $C_1$ and $C_2$ be subgraphs of $AT(n,p)$ induced by $c_1$ and $c_2$, respectively.

Obviously, $AT(n,p)$ has $p+1$ layers of vertices and every layer has two types of vertices, i.e. type $0$ and type $1$. In the figures the vertices of type $0$ always lie lower than the vertices of type $1$. The set of vertices of type $k$ in layer $i$ is denoted by $V^k_i$. Moreover, let the vertices in $V^k_i$ be denoted as follows: $V^k_i = \lbrace v^k_{i,0}, \ldots, v^k_{i,n-1} \rbrace$. See Figure \ref{armchair} for an example.

\begin{figure}[!htb]
	\centering
		\includegraphics[scale=0.8, trim=0cm 0cm 1cm 0cm]{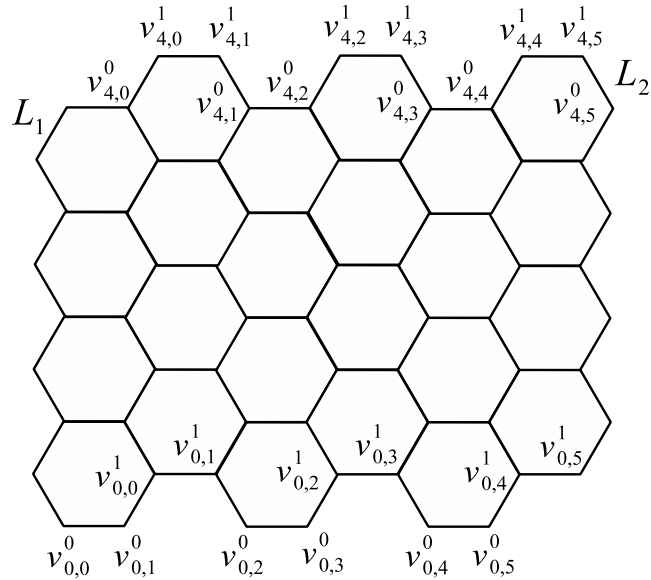}
\caption{Armchair tubulene $AT(6,4)$ with vertices in $V^0_0$, $V^1_0$, $V^0_4$, and $V^1_4$. Curves $L_1$ and $L_2$ are joined together.}
	\label{armchair}
\end{figure}


In this section, we determine the orbits under the natural action of the group ${\aut}(AT(n,p))$ on the set $V(AT(n,p))$. First, one lemma is needed.

\begin{lemma}
\label{le2}
Let $\varphi: V(C_1) \rightarrow V(C_i)$ be an isomorphism between subgraphs $C_1$ and $C_i$, where $i \in \lbrace 1,2 \rbrace$. Then there is exactly one automorphism $\overline{\varphi}: V(AT(n,p)) \rightarrow V(AT(n,p))$ such that $\varphi(x) = \overline{\varphi}(x)$ for any $x \in V(C_1)$.
\end{lemma}

\begin{proof}
Let $\varphi: V(C_1) \rightarrow V(C_i)$ be an isomorphism where $i \in \lbrace 1,2 \rbrace$. For any $x \in V(C_1) = V^0_0 \cup V^1_0$ we define $\overline{\varphi}(x) = \varphi(x)$. In the rest of the proof we will define function $\overline{\varphi}$ step by step such that every edge will be mapped to an edge and $\overline{\varphi}$ will be a bijection.

First let $x \in V^0_1$. Then there is exactly one $y \in V^1_0$ such that $x$ and $y$ are adjacent. Since the degree of $y$ is 3, let $y_1$ and $y_2$ be the other two neighbours of $y$ in $AT(n,p)$. Obviously, $\overline{\varphi}(y), \overline{\varphi}(y_1)$, and $\overline{\varphi}(y_2)$ are already defined and it holds that $\overline{\varphi}(y_1)$ and $\overline{\varphi}(y_2)$ are both adjacent to $\overline{\varphi}(y)$. Since the degree of $\overline{\varphi}(y)$ is 3, we define $\overline{\varphi}(x)$ to be the neighbour of $\overline{\varphi}(y)$, different from $\overline{\varphi}(y_1)$ and $\overline{\varphi}(y_2)$. This can be done for any $x \in V^0_1$.

Now let $x \in V^1_1$. Then there is exactly one vertex $y \in V^0_1$ such that $y$ is adjacent to $x$. Let $y_1$ and $y_2$ be the other two neighbours of $y$. It is easy to see that $\overline{\varphi}(y)$, $\overline{\varphi}(y_1)$, and $\overline{\varphi}(y_2)$ are already defined. Also, the degree of $\overline{\varphi}(y)$ is $3$. Therefore, we define $\overline{\varphi}(x)$ to be the neighbour of $\overline{\varphi}(y)$, different from $\overline{\varphi}(y_1)$ and $\overline{\varphi}(y_2)$. This can be done for any $x \in V^1_1$. 

With the procedure above we have defined function $\overline{\varphi}$ on the set of vertices $V^0_0 \cup V^1_0 \cup V^0_1 \cup V^1_1$ such that for any two adjacent vertices $x,y \in V^0_0 \cup V^1_0 \cup V^0_1 \cup V^1_1$, it holds that $\overline{\varphi}(x)$ and $\overline{\varphi}(y)$ are also adjacent. Using induction, we can define function $\overline{\varphi}$ on the set $V(AT(n,p))$ such that for any two adjacent vertices $x,y \in V(AT(n,p))$ it holds that $\overline{\varphi}(x)$ and $\overline{\varphi}(y)$ are adjacent. Since $\overline{\varphi}$ is also bijective, it is an automorphism of the graph $AT(n,p)$. It follows from the construction that $\overline{\varphi}$ is also unique. Therefore, the proof is complete. \qed
\end{proof}

Finally, we obtain the orbits under the natural action of the group ${\aut}(AT(n,p))$ on   the set $V(AT(n,p))$. 
\begin{theorem}
The orbits under the natural action of the group ${\aut}(AT(n,p))$ on   the set $V(AT(n,p))$ are:
\begin{itemize}
\item [1.] if $p$ is odd
$$O^0_i = V^0_i \cup V^1_{n-i}, \ i \in \Big\lbrace 0, \ldots, \frac{p-1}{2} \Big\rbrace, $$
$$O^1_i = V^1_i \cup V^0_{n-i}, \ i \in \Big\lbrace 0, \ldots, \frac{p-1}{2}  \Big\rbrace. $$
\item [2.] if $p$ is even

$$O^0_i = V^0_i \cup V^1_{n-i}, \ i \in \Big\lbrace 0, \ldots, \frac{p-2}{2} \Big\rbrace, $$
$$O^1_i = V^1_i \cup V^0_{n-i}, \ i \in \Big\lbrace 0, \ldots, \frac{p-2}{2}  \Big\rbrace, $$
$$O_{\frac{p}{2}} = V^0_{\frac{p}{2}} \cup V^1_{\frac{p}{2}}.$$
\end{itemize}

\end{theorem}

\begin{proof}
It follows from the proof of Lemma \ref{le2} that for any vertex $x$ of type $k$ in layer $i$, where $i \in \lbrace 0, \ldots, p \rbrace$, $k \in \lbrace 0,1 \rbrace$, and any vertex $y$ in layer $i$ of type $k$ or in layer $n-i$ of type $1-k$, there is an automorphism that maps $x$ to $y$. We notice that this also works when $p$ is even and $i = \frac{p}{2}$, which means that if $x \in V^0_{\frac{p}{2}}$ and $y \in V^1_{\frac{p}{2}}$, there is an automorphism that maps $x$ to $y$. 

Also, if $x$ is in layer $i$ and $y$ is in layer $j$, $j \neq i, j \neq n-i$, the distance from $x$ to $C_1$ or $C_2$, i.e. $\min \lbrace d(x,C_1),d(x,C_2) \rbrace$, can not be the same as the distance from $y$ to $C_1$ or $C_2$, i.e. $\min \lbrace d(y,C_1),d(y,C_2) \rbrace$. Therefore, there is no automorphism that maps $x$ to $y$. 

Moreover, if $x \in V^0_i$ and  $y \in V^1_i$ or  $y \in V^0_{n-i}$, where $i \in \lbrace 0, \ldots, p \rbrace$, $i \neq \frac{p}{2}$, then the numbers $\min \lbrace d(x,C_1),d(x,C_2) \rbrace$ and $\min \lbrace d(y,C_1),d(y,C_2) \rbrace$ can not be the same. Again, there is no automorphism that maps $x$ to $y$. 

Therefore, the proof is complete. \qed
\end{proof}

Lemma \ref{le2} claims that any isomorphism between subgraphs $C_1$ and $C_i$, where $i \in \lbrace 1,2 \rbrace$, can be extended to the automorphism of the graph $AT(n,p)$. In the next proposition we show the other direction.

\begin{proposition}
\label{le1}
Let $\varphi: V(AT(n,p)) \rightarrow V(AT(n,p))$ be an automorphism. Then the function $\varphi': V(C_1) \rightarrow \varphi(V(C_1))$, $\varphi'(x) =\varphi(x)$ for $x \in V(C_1)$, defines an automorphism of $C_1$ or an isomorphism from $C_1$ to $C_2$. 
\end{proposition}

\begin{proof}
The graph $AT(n,p)$ contains exactly two disjoint cycles of length $2n$ with exactly $n$ vertices of degree 2 in the graph $AT(n,p)$. These two are $C_1$ and $C_2$. Therefore, automorphism $\varphi$ maps $C_1$ to either $C_1$ or $C_2$ and the proof is complete. \qed
\end{proof}

Hence, we obtain that all the automorphisms of graph $AT(n,p)$ can be obtained by finding all the automorphisms of subgraph $C_1$ and all the isomorphisms from subgraph $C_1$ to subgraph $C_2$. It is easy to see that the automorphism group of subgraph $C_1$ is isomorphic to the dihedral group $D_{\frac{n}{2}}$. Moreover, any isomorphism from $C_1$ to $C_2$ can be obtained as the composition of an automorphism of subgraph $C_1$ and a fixed isomorphism from $C_1$ to $C_2$. Therefore, we state the following conjecture.
  
\begin{conjecture}
Let $AT(n,p)$ be an armchair tubulene. The automorphism group of the graph $AT(n,p)$ is isomorphic to the direct product of the dihedral group $D_{\frac{n}{2}}$ and the cyclic group $\mathbb{Z}_2$.
\end{conjecture}

\section{The Graovac-Pisanski index of armchair tubulenes}

In this section, we calculate the Graovac-Pisanski index of armchair tubulenes. We have to consider the following four cases. The first part is explained in details, while for the remaining cases only the important results are given. We always denote by $u$ an arbitrary element of $V^0_0$ and by $v$ an arbitrary element of $V^1_0$.

\begin{enumerate}
\item $p$ is even and $4 \, | \, n$ \\
It is enough to compute $W(O^0_0)$ and $W(O^1_0)$ since, for example, $W(O^0_1)$ of the graph $AT(n,p)$ is exactly $W(O^0_0)$ of the graph $AT(n,p-2)$ (the graph $AT(n,p-2)$ is a convex subgraph of the graph $AT(n,p)$). Beside that, we need to calculate $W(O_{\frac{p}{2}})$. Since the graph induced on the vertices in $O_{\frac{p}{2}}$ is an isometric cycle of length $2n$, we have $W(O_{\frac{p}{2}})=n^3$. 

Next, we need to calculate $d(u,V^0_0)$ and therefore, we consider distances between some vertices on the cycle of length $2n$, see Figure \ref{cikel}. Note that the thick vertices represent the vertices in set $V^0_0$. Therefore,

\begin{eqnarray*}
d(u,V^0_0) & = & \sum_{i=0}^{\frac{n}{4}-1}(3+4i) + \sum_{i=0}^{\frac{n}{4}-1}(4+4i) + \sum_{i=0}^{\frac{n}{4}-1}(1+4i) + \sum_{i=0}^{\frac{n}{4}-2}(4+4i) \\
& = & \frac{n^2}{2}.
\end{eqnarray*}

\begin{figure}[h!] 
\begin{center}
\includegraphics[scale=0.7]{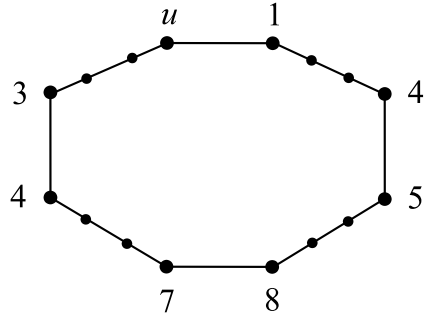}
\end{center}
\caption{\label{cikel} Subgraph $C_1$ when $n=8$ with the distances from vertices in $V^0_0$ to $u$.}
\end{figure}

Obviously, $d(v,V^1_0) = d(u,V^0_0) = \frac{n^2}{2}$. To determine $d(u,V^1_p)$ and $W(O^0_0)$, we consider two cases.

\smallskip

\begin{enumerate}

\item $n \leq 4p + 4$ \\
In this case, we can draw two lines $a$ and $b$, see Figure \ref{razdalje2}. All $n$ vertices of $V^1_p$ are between lines $a$ and $b$ or near lines $a$ and $b$ (at most $4$ vertices).

\begin{figure}[h]
	\centering
		\includegraphics[scale=0.7, trim=0cm 0cm 1cm 0cm]{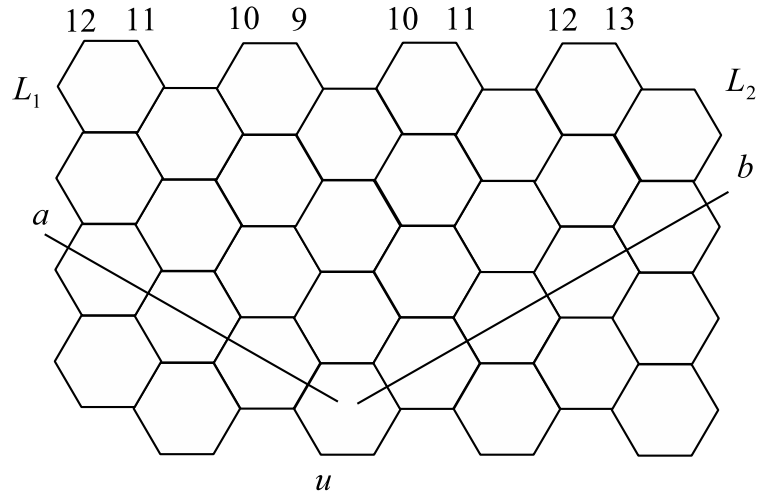}
\caption{Distances from $u$ in $AT(8,4)$. Curves $L_1$ and $L_2$ are joined.}
	\label{razdalje2}
\end{figure}

It is easy to observe that a shortest path from vertex $u$ to some vertex $x \in V^1_p$ can be obtained by joining a path following line $a$ or line $b$ and a vertical path. Therefore, the distance from $u$ to the vertex directly above $u$ equals $2p+1$ and the distance increases by $1$ for every next vertex in $V^1_p$ (in both directions). For an example see Figure \ref{razdalje2}. Hence, we get
 
\begin{eqnarray*}
d(u,V^1_p) & = & (2p+1) + 2 \sum_{i=1}^{\frac{n-2}{2}}(2p+1+i) + \Big(2p+1 + \frac{n}{2}\Big) \\
& = & \frac{n^2}{4} + 2np + n.
\end{eqnarray*}

Therefore,
$$d(u,O^0_0) = d(u,V^0_0) + d(u,V^1_p) = \frac{3n^2}{4} + 2np + n$$
and since every vertex in $O^0_0$ has equivalent position, we deduce

\begin{eqnarray*}
W(O^0_0) & = & \frac{1}{2} \cdot |O^0_0| \cdot d(u,O^0_0) =  \frac{2n}{2}\Big(\frac{3n^2}{4} + 2np + n\Big) \\
& = & n\Big(\frac{3n^2}{4} + 2np + n\Big).
\end{eqnarray*}

\bigskip

\item  $n > 4p + 4$\\
In this case, we also draw two lines $a$ and $b$ as before. There are exactly $4p$ vertices of $V^1_p$ between lines $a$ and $b$, exactly $4$ vertices ($2$ on each side) of $V^1_p$  near lines $a$ and $b$, and $n - 4p - 4$ other vertices. See Figure \ref{razdalje1}.

\begin{figure}[h]
	\centering
		\includegraphics[scale=0.7, trim=0cm 0cm 1cm 0cm]{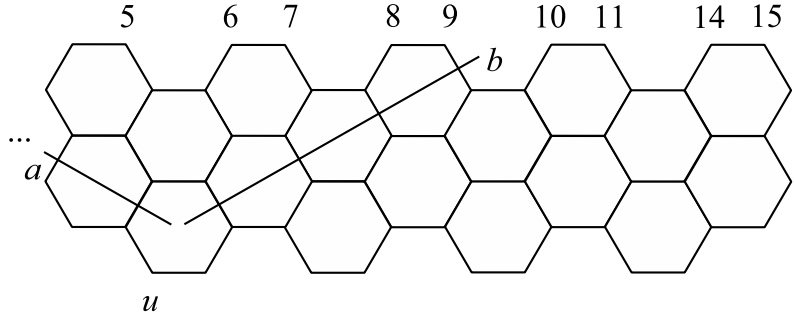}
\caption{Distances from $u$ in one part of an armchair tubulene.}
	\label{razdalje1}
\end{figure}

We can notice that the distance from $u$ to the vertex directly above $u$ is $2p+1$ and that the distance from $u$ increases by $1$ (in both directions) for every next vertex among other $4p+3$ vertices that are between or near lines $a$ and $b$. Afterwards, for the rest $n-4p-4$ vertices the increase of the distance from $u$ alternates between $3$ and $1$ in both directions. Therefore, we get

\begin{eqnarray*}
d(u,V^1_p) & = & (2p+1) + 2 \sum_{i=1}^{\frac{4p+2}{2}}(2p+1+i) + (2p+1 + 2p+2) \\
& + & \sum_{i=0}^{\frac{n-4p-8}{4}}(4p+5 + 4i) + \sum_{i=0}^{\frac{n-4p-8}{4}}(4p+6 + 4i) \\
&+ & \sum_{i=0}^{\frac{n-4p-8}{4}}(4p+6 + 4i) + \sum_{i=0}^{\frac{n-4p-8}{4}}(4p+7 + 4i) \\
& = & \frac{n^2}{2} + p(4p+4).
\end{eqnarray*}

Consequently,
$$d(u,O^0_0) = d(u,V^0_0) + d(u,V^1_p) =n^2 + 4p^2 + 4p$$
and since every vertex in $O^0_0$ has equivalent position, we obtain

\begin{eqnarray*}
W(O^0_0) & = & \frac{1}{2} \cdot |O^0_0| \cdot d(u,O^0_0) =  \frac{2n}{2}(n^2 + 4p^2 + 4p) \\
& = & n(n^2 + 4p^2 + 4p).
\end{eqnarray*}
\end{enumerate}
\bigskip

%


\noindent To compute $W(O^1_0)$, we also consider two cases.
\smallskip

\begin{enumerate}
\item $n \leq 4p$ \\
Similar as before, we can draw two lines $a$ and $b$ as shown in Figure \ref{razdalje3}. All $n$ vertices of $V^0_p$ are between lines $a$ and $b$ or near the lines $a$ and $b$ (at most $4$ vertices).

\begin{figure}[h]
	\centering
		\includegraphics[scale=0.7, trim=0cm 0cm 1cm 0cm]{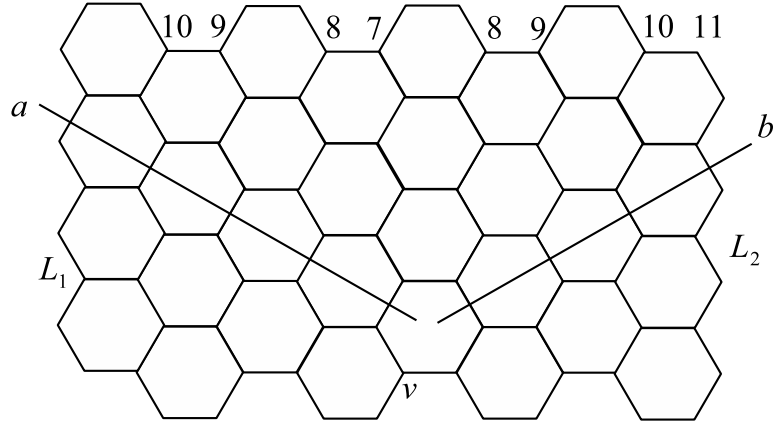}
\caption{Distances from $v$ in $AT(8,4)$. Curves $L_1$ and $L_2$ are joined.}
	\label{razdalje3}
\end{figure}

It is easy to observe that the distance from vertex $v$ to the vertex directly above $v$ equals $2p-1$ and that the distance increases by $1$ for every next vertex in $V^0_p$ (in both directions). For an example see Figure \ref{razdalje3}. Hence, we get
 
\begin{eqnarray*}
d(v,V^0_p) & = & (2p-1) + 2 \sum_{i=1}^{\frac{n-2}{2}}(2p-1+i) + \Big(2p-1 + \frac{n}{2}\Big) \\
& = & \frac{n^2}{4} + 2np - n.
\end{eqnarray*}

Therefore,
$$d(v,O^1_0) = d(v,V^1_0) + d(v,V^0_p) = \frac{3n^2}{4} + 2np - n$$
and since every vertex in $O^1_0$ has equivalent position, we deduce

\begin{eqnarray*}
W(O^1_0) & = & \frac{1}{2} \cdot |O^1_0| \cdot d(v,O^1_0) =  \frac{2n}{2}\Big(\frac{3n^2}{4} + 2np - n\Big) \\
& = & n\Big(\frac{3n^2}{4} + 2np - n\Big).
\end{eqnarray*}

\bigskip

\item $n > 4p$ \\
In this case, we also draw two lines $a$ and $b$ as in the previous case. There are exactly $4p-4$ vertices of $V^0_p$ between lines $a$ and $b$, exactly $4$ vertices ($2$ on each side) of $V^0_p$  near lines $a$ and $b$, and $n - 4p$ other vertices. 

We can notice that the distance from $v$ to the vertex directly above $v$ is $2p-1$ and that the distance from $v$ increases by $1$ (in both directions) for every next vertex among other $4p-1$ vertices that are between or near lines $a$ and $b$. Afterwards, for the rest $n-4p$ vertices the increase of the distance from $v$ alternates between $3$ and $1$ in both directions. Therefore, we get

\begin{eqnarray*}
d(v,V^0_p) & = & (2p-1) + 2 \sum_{i=1}^{\frac{4p-2}{2}}(2p-1+i) + (2p-1 + 2p) \\
& + & \sum_{i=0}^{\frac{n-4p-4}{4}}(4p+1 + 4i) + \sum_{i=0}^{\frac{n-4p-4}{4}}(4p+2 + 4i) \\
&+ & \sum_{i=0}^{\frac{n-4p-4}{4}}(4p+2 + 4i) + \sum_{i=0}^{\frac{n-4p-4}{4}}(4p+3 + 4i) \\
& = & \frac{n^2}{2} + p(4p-4).
\end{eqnarray*}

Consequently,
$$d(v,O^1_0) = d(v,V^1_0) + d(v,V^0_p) =n^2 + 4p^2 - 4p$$
and since every vertex in $O^1_0$ has equivalent position, we get

\begin{eqnarray*}
W(O^1_0) & = & \frac{1}{2} \cdot |O^1_0| \cdot d(v,O^1_0) =  \frac{2n}{2}(n^2 + 4p^2 - 4p) \\
& = & n(n^2 + 4p^2 - 4p).
\end{eqnarray*}
\end{enumerate}
\bigskip

\noindent Putting all the results together, we obtain Table \ref{tabela1}.
%
%

\begin{center}
\begin{table}[H] \renewcommand{\arraystretch}{1.5}
\begin{tabular}{|c||c|c|}

\hline 
& $n \leq 4p+4$ & $n > 4p+4$ \\
\hline \hline
$d(u,V^0_0)$ & $\frac{n^2}{2}$ & $\frac{n^2}{2}$  \\ 
 \hline 
$d(u,V^1_p)$ & $\frac{n^2}{4} + 2np + n$ & $\frac{n^2}{2} + 4p^2 + 4p$  \\ 
\hline 
$d(u,O^0_0)$ & $\frac{3n^2}{4} + 2np + n$ & $n^2 + 4p^2 + 4p$  \\ 
\hline 
$W(O^0_0)$ & $n\Big(\frac{3n^2}{4} + 2np + n\Big)$ & $n(n^2 + 4p^2 + 4p)$  \\ 
\hline \hline
& $n \leq 4p$ & $n > 4p$ \\
\hline \hline
 $d(v,V^1_0)$ & $\frac{n^2}{2}$  & $\frac{n^2}{2}$  \\ 
\hline 
$d(v,V^0_p)$ & $\frac{n^2}{4} + 2np - n$   & $\frac{n^2}{2} + 4p^2-4p$  \\ 
\hline 
$d(v,O^1_0)$ & $\frac{3n^2}{4} + 2np - n$  & $n^2 + 4p^2-4p$ \\ 
\hline 
$W(O^1_0)$ & $n\Big(\frac{3n^2}{4} + 2np - n\Big)$ & $n(n^2 + 4p^2 - 4p)$  \\
\hline 
\end{tabular} 
\caption{\label{tabela1} Distances in $AT(n,p)$ with $p$ even and $4 \, | \, n$.}
\end{table}
\end{center}

To compute $\widehat{W}(AT(n,p))$, we use Formula \ref{formula}. First define the following functions:
$$\begin{array}{rcl}
f_1(p) & = &  n\Big(\frac{3n^2}{4} + 2np + n\Big), \\
f_2(p) & = &  n(n^2 + 4p^2 + 4p), \\
g_1(p) & = & n\Big(\frac{3n^2}{4} + 2np - n\Big), \\
g_2(p) & = & n(n^2 + 4p^2 - 4p). \\

\end{array}$$

\noindent One can easily notice that $W(O^0_i)=f_1(p-2i)$ if $n \leq 4(p-2i)+4 = 4p-8i+4$ and $W(O^0_i)=f_2(p-2i)$ if $n > 4p-8i+4 $ (and similar can be done for $W(O^1_i)$). Now consider the following four cases. 
\begin{itemize}
\item [(a)] $n > 4p+4$ \\
\noindent It follows
$$ {W'}(AT(n,p))  =  W(O_{\frac{p}{2}}) + \sum_{i=1}^{\frac{p}{2}}f_2(2i) + \sum_{i=1}^{\frac{p}{2}}g_2(2i).$$

\item [(b)] $n = 4p+4$ \\
\noindent For $p \geq 4$ it follows
$$ {W'}(AT(n,p))  = W(O_{\frac{p}{2}}) + \sum_{i=1}^{\frac{p-2}{2}}f_2(2i) + f_1(p) + \sum_{i=1}^{\frac{p}{2}}g_2(2i).$$
The case $p=2$ can be easily computed in a similar way.

\item [(c)] $n \leq 4p$ and $8 \, | \, n$ \\
\noindent For $n \geq 16$ it follows
\begin{eqnarray*} 
{W'}(AT(n,p))  & = & W(O_{\frac{p}{2}}) + \sum_{i=1}^{\frac{n-8}{8}}f_2(2i) + \sum_{i=\frac{n}{8}}^{\frac{p}{2}}f_1(2i) \\
& + & \sum_{i=1}^{\frac{n-8}{8}}g_2(2i) + \sum_{i=\frac{n}{8}}^{\frac{p}{2}}g_1(2i).
\end{eqnarray*}
The case $n=8$ can be easily computed in a similar way.

\item [(d)] $n \leq 4p$ and $8 \, | \, (n-4)$ \\
\noindent For $n\geq 20$ it follows
\begin{eqnarray*} 
{W'}(AT(n,p))  & = & W(O_{\frac{p}{2}}) +  \sum_{i=1}^{\frac{n-12}{8}}f_2(2i) + \sum_{i=\frac{n-4}{8}}^{\frac{p}{2}}f_1(2i) \\
& + & \sum_{i=1}^{\frac{n-4}{8}}g_2(2i) + \sum_{i=\frac{n+4}{8}}^{\frac{p}{2}}g_1(2i).
\end{eqnarray*}
The cases $n=12$ or $n=4$ can be easily computed in a similar way.

\end{itemize}
To compute all the sums from the previous cases, we use a computer program. Since $|V(AT(n,p))| = 2n(p+1)$ and the cardinality of any orbit of $AT(n,p)$ is $2n$, it is easy to see that $\widehat{W}(AT(n,P)) = (p+1)W'(AT(n,P))$. The results are presented in the first part of Table \ref{tabela5}.
\bigskip

\item $p$ is even and $4 \, | \, (n-2)$ \\
All the details are similar to the case 1. Therefore, the important results are presented in Table \ref{tabela2}. We also have $W(O_{\frac{p}{2}})=n^3$. The values of the Graovac-Pisanski index in this case are shown in the second part of Table \ref{tabela5}.

\begin{center}
\begin{table}[H] \renewcommand{\arraystretch}{1.5}
\begin{tabular}{|c||c|c|}

\hline 
& $n \leq 4p+4$ & $n > 4p+4$ \\
\hline \hline
$d(u,V^0_0)$ & $\frac{n^2-2}{2}$ & $\frac{n^2-2}{2}$  \\ 
 \hline 
$d(u,V^1_p)$ & $\frac{n^2}{4} + 2np + n$ & $\frac{n^2}{2} + 4p^2 + 4p + 1$  \\ 
\hline 
$d(u,O^0_0)$ & $\frac{3n^2}{4} + 2np + n - 1$ & $n^2 + 4p^2 + 4p$  \\ 
\hline 
$W(O^0_0)$ & $n\Big(\frac{3n^2}{4} + 2np + n - 1\Big)$ & $n(n^2 + 4p^2 + 4p)$  \\ 
\hline \hline
& $n \leq 4p$ & $n > 4p$ \\
\hline \hline
 $d(v,V^1_0)$ & $\frac{n^2-2}{2}$  & $\frac{n^2-2}{2}$  \\ 
\hline 
$d(v,V^0_p)$ & $\frac{n^2}{4} + 2np - n$   & $\frac{n^2}{2} + 4p^2-4p + 1$  \\ 
\hline 
$d(v,O^1_0)$ & $\frac{3n^2}{4} + 2np - n - 1$  & $n^2 + 4p^2-4p$ \\ 
\hline 
$W(O^1_0)$ & $n\Big(\frac{3n^2}{4} + 2np - n - 1\Big)$ & $n(n^2 + 4p^2 - 4p)$  \\
\hline 
\end{tabular} 
\caption{\label{tabela2} Distances in $AT(n,p)$ with $p$ even and $4 \, | \, (n-2)$.}
\end{table}
\end{center}

\item $p$ is odd and $4 \, | \, n$ \\
All the details are similar to the case 1. It turns out that the distances are the same as for even $p$. Therefore, we can consider  Table \ref{tabela1}. The values of the Graovac-Pisanski index in this case are shown in the third part of Table \ref{tabela5}.

\item $p$ is odd and $4 \, | \, (n-2)$ \\
All the details are similar to the case 1. As above it turns out that the distances are the same as for even $p$. Therefore, we can consider  Table \ref{tabela2}. The values of the Graovac-Pisanski index in this case are shown in the last part of Table \ref{tabela5}.

\end{enumerate}

\noindent
Finally, the results for the Graovac-Pisanski index of $AT(n,p)$ are shown in Table \ref{tabela5}. The results for some small cases are omitted.

\begin{center}
\begin{table}[H] \renewcommand{\arraystretch}{1.5}
\begin{tabular}{|c||c|}

\hline 
$p$ even and $4 \, | \, n$ & \\
\hline 
$n > 4p+4$ & $(p+1) \big(n^3p + n^3 + \frac{4np^3}{3} + 4np^2 + \frac{8np}{3} \big)$   \\ 
 \hline 
$n=4p+4, p \geq 4$ & $(p+1) \big(n^3p + \frac{3n^3}{4} + 2n^2p + n^2 + \frac{4np^3}{3} - \frac{4np}{3} \big)$   \\ 
\hline 
$n \leq 4p$, $n \geq 16$ & $(p+1) \big(\frac{n^4}{48} + \frac{3n^3p}{4} + \frac{3n^3}{4} + n^2p^2 + 2n^2p + \frac{2n^2}{3} \big)$  \\ 

\hline \hline
$p$ even and $4 \, | \, (n-2)$ & \\
\hline 
$n > 4p+4$ & $(p+1) \big(n^3p + n^3 + \frac{4np^3}{3} + 4np^2 + \frac{8np}{3} \big)$   \\ 
 \hline 
$n=4p+2, p \geq 4$ & $(p+1) \big(n^3p + \frac{3n^3}{4} + 2n^2p + n^2 + \frac{4np^3}{3} - \frac{4np}{3} - n \big)$   \\ 
\hline 
$n \leq 4p$, $n \geq 14$ & $(p+1) \big(\frac{n^4}{48} + \frac{3n^3p}{4} + \frac{3n^3}{4} + n^2p^2 + 2n^2p + \frac{11n^2}{12} - np - n \big)$  \\ 

\hline \hline
$p$ odd and $4 \, | \, n$ & \\
\hline 
$n > 4p+4$ &   $(p+1) \big(n^3p + n^3 + \frac{4np^3}{3} + 4np^2 + \frac{8np}{3} \big)$  \\ 
 \hline 
$n=4p+4, p \geq 3$ & $(p+1) \big(n^3p + \frac{3n^3}{4} + 2n^2p + n^2 + \frac{4np^3}{3} - \frac{4np}{3} \big)$    \\ 
\hline 
$n \leq 4p$, $n \geq 12$ & $(p+1) \big(\frac{n^4}{48} + \frac{3n^3p}{4} + \frac{3n^3}{4} + n^2p^2 + 2n^2p + \frac{2n^2}{3} \big)$  \\ 

\hline \hline
$p$ odd and $4 \, | \, (n-2)$ & \\
\hline 
$n > 4p+4$ &   $(p+1) \big(n^3p + n^3 + \frac{4np^3}{3} + 4np^2 + \frac{8np}{3} \big)$  \\ 
 \hline 
$n=4p+2, p \geq 3$ & $(p+1) \big(n^3p + \frac{3n^3}{4} + 2n^2p + n^2 + \frac{4np^3}{3} - \frac{4np}{3} - n \big)$  \\ 
\hline 
$n \leq 4p$, $n \geq 10$ & $(p+1) \big(\frac{n^4}{48} + \frac{3n^3p}{4} + \frac{3n^3}{4} + n^2p^2 + 2n^2p + \frac{11n^2}{12} - np - n \big)$ \\ 
\hline

\end{tabular} 
\caption{\label{tabela5} Closed formulas for $\widehat{W}(AT(n,p))$.}
\end{table}
\end{center}

\section*{Acknowledgment} 

\noindent The author Petra \v Zigert Pleter\v sek acknowledge the financial support from the Slovenian Research Agency (research core funding No. P1-0297). 

\noindent The author Niko Tratnik was financially supported by the Slovenian Research Agency.

\end{document}